\documentclass[11pt,reqno]{amsart}

\usepackage[a4paper,margin=3cm]{geometry}
  \usepackage[T1]{fontenc}
  \usepackage[utf8]{inputenc}
\usepackage{microtype}
\usepackage{graphicx}
  \usepackage{amssymb,mathtools,amsthm}
\usepackage{tikz}
\usepackage{caption}
\usepackage{subcaption}

\usepackage{hyperref}
\definecolor{colorlinks}{RGB}{0, 24, 168}
\definecolor{colorcites}{RGB}{124, 10, 2}
\hypersetup{
    colorlinks=true,
    linkcolor=colorlinks,
    citecolor=colorcites,
    urlcolor=colorlinks,
    pdfborder={0 0 0}
}

\newcounter{n}
\numberwithin{n}{section}
\theoremstyle{plain}
  \newtheorem{lemma}[n]{Lemma}
  \newtheorem{theorem}[n]{Theorem}
  \newtheorem*{theorem*}{Theorem}
  
  \newtheorem{corollary}[n]{Corollary}
\theoremstyle{definition}
  \newtheorem{definition}[n]{Definition}

\renewcommand\phi\varphi

\author{Piet Lammers}
\address{Institut des Hautes \'Etudes Scientifiques}
\email{lammers@ihes.fr}

\author{S\'{e}bastien Ott}
\address{D\'epartement de Math\'ematiques, Universit\'e de Fribourg}
\email{ott.sebast@gmail.com}

\title[Delocalisation and absolute-value-FKG in the SOS model]{Delocalisation and absolute-value-FKG\\in the solid-on-solid model}

 \makeatletter
 \@namedef{subjclassname@2020}{\textup{2020} Mathematics Subject Classification}
 \makeatother


\begin{document}
\begin{abstract}
The solid-on-solid model is a model of height functions,
introduced to study the interface separating the $+$ and $-$ phase in the Ising model.
The planar solid-on-solid model thus corresponds to the three-dimensional Ising model.
Delocalisation of this model at high temperature and at zero slope was first derived by Fr\"ohlich and Spencer,
in parallel to proving the Berezinskii-Kosterlitz-Thouless phase transition.
The first main result of this article consists of a simple, alternative proof of delocalisation of the solid-on-solid model.
In fact, the argument is more general: it works on any planar graph---not just the square lattice---and implies that the interface delocalises at any slope rather than exclusively at the zero slope.
The second main result, proved independently, is that the absolute value of the height function in this model satisfies the FKG lattice condition.
This property is believed to be intimately linked to the (quantitative) understanding of delocalisation,
given the recent successes in the context of the square ice and (more generally) the six-vertex model, and it has already been used elsewhere in a new proof of the BKT transition.
The new FKG inequality is shown to hold true for both the solid-on-solid model as well as for the discrete Gaussian model,
which in this article implies that the two notions of delocalisation,
namely delocalisation in finite volume and delocalisation of shift-invariant Gibbs measures,
coincide.

\end{abstract}
\maketitle


\section{Introduction}

\subsection{Background on phase coexistence and interfaces}

The analysis of phase coexistence has a central role in classical statistical physics.
The phenomenon is omnipresent in both natural systems at the transition point
(such as for example the liquid-gas transition of water),
as well as in the (idealised) mathematical models designed to study these systems.
The simplest and perhaps most well-known such model is the Ising model,
which has been studied extensively since its inception in 1920.
When two phases (are forced to) coexist,
one can study a derived object, namely the interface separating the two phases.
Few rigorous results on the macroscopic behaviour of the interface in the Ising model are known,
and it is expected that this behaviour depends strongly 
on three fundamental parameters:
the dimension,
the inverse temperature,
and the inclination at which boundary conditions are enforced.

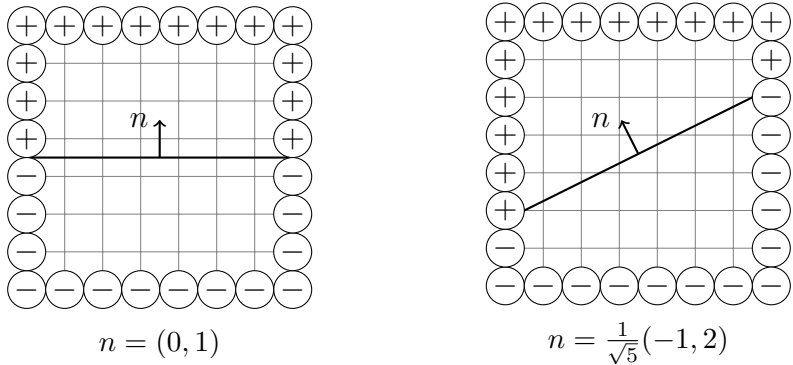
\begin{figure}
	\centering
  \newcommand{\fw}{.4\textwidth}
  \newcommand\sign[3]{
  	\draw[fill=white] (#1,#2) circle (.5);
  	\node at (#1,#2) {$#3$};
  }
	\begin{subfigure}{\fw}
		\centering
		\begin{tikzpicture}[scale=0.5]
		\draw[very thin, gray] (0,0) grid (7,7);
		\draw[thick] (0,3.5)--(7,3.5);
		\draw[thick, ->] (3.5,3.5)--(3.5,4.5);
		\draw (3.5,4.5) node[left] {$n$};
		\foreach \x in {4,5,6,7}{
		    \sign{0}{\x}{+}
		    \sign{7}{\x}{+}
		}
		\foreach \x in {0,1,2,3}{
		    \sign{0}{\x}{-}
		    \sign{7}{\x}{-}
		}
		\foreach \x in {1,2,3,4,5,6}{
			\sign{\x}{7}{+}
			\sign{\x}{0}{-}
		}
		\end{tikzpicture}
		\caption*{\(n=(0,1)\)}
	\end{subfigure}
  \begin{subfigure}{\fw}
		\centering
		\begin{tikzpicture}[scale=0.5]
		\draw[very thin, gray] (0,0) grid (7,7);
		\draw[thick, black] (0,1.75)--(7,5.25);
		\draw[thick, ->] (3.5,3.5)--(3.0527864045 , 4.394427191 );
		\draw (3.0527864045 , 4.394427191 ) node[left] {$n$};
		\foreach \x in {2,3,4,5,6,7}{
		    \sign{0}{\x}{+}
		    \sign{7}{7-\x}{-}
		}
		\foreach \x in {0,1}{
		    \sign{0}{\x}{-}
		    \sign{7}{7-\x}{+}
		}
		\foreach \x in {1,2,3,4,5,6}{
		    \sign{\x}{7}{+}
		    \sign{\x}{0}{-}
		}
		\end{tikzpicture}
		\caption*{\(n=\frac{1}{\sqrt{5}}(-1,2)\)}
	\end{subfigure}
	\caption{Dobrushin boundary conditions and the interface normal vector}
	\label{fig:Ising_tilt_Dobr}
\end{figure}

To illustrate the rich behaviour of the interface, we shall give a concise, informal account of the 
expected phase diagram.
Let $d\in\mathbb N$ denote the dimension of the underlying square lattice $\mathbb Z^d$,
write $\beta_c=\beta_c(d)>0$ for the critical inverse temperature,
let $\beta>\beta_c$ denote the inverse temperature of interest,
and let $n\in S^{d-1}$ denote the interface normal vector for the Dobrushin boundary condition
(see Figure~\ref{fig:Ising_tilt_Dobr}).
Naturally, the interface is of dimension $d-1$.
We shall fix boundary conditions outside a box of sides $N$,
and consider the variance of the ``height'' of the interface at the centre vertex.
This number---the variance---is called the \emph{fluctuation} of the interface.
The expected behaviour is as follows.
\begin{itemize}
	\item In dimension $d=2$, the fluctuations are always of order $N$ for all temperatures
	and for all inclinations.
	The interface is one-dimensional, and its macroscopic behaviour is therefore 
	linked to the theory of Markov chains and Brownian motion.
	\item In dimension $d=3$, the expected behaviour depends on $\beta$ and $n$.
	For $n=(0,0,1)$ and for $\beta$ large, the fluctuations are uniformly bounded,
	while for $\beta$ close to $\beta_c$ or for $n$ not parallel to an edge of the underlying lattice,
	the fluctuations are expected to be of order $\log N$.
	\item In dimension $d\geq 4$, the fluctuations should always be uniformly bounded. 
\end{itemize}
When the fluctuations are bounded uniformly, it is said that the interface
is \emph{localised} or \emph{smooth},
and the interface is said to be \emph{delocalised} or \emph{rough} otherwise.
The two-dimensional case is well-understood
with a full invariance principle towards Brownian motion for all \(\beta>\beta_c\);
see~\cite{Campanino+Ioffe+Velenik-2003} and the references therein. 
In three dimensions, the only available result is that of Dobrushin~\cite{Dobrushin-1972} who proved (in particular)
that fluctuations are uniformly bounded in \(N\) whenever \(n=(0,0,1)\) and for \(\beta\) sufficiently large;
see also~\cite{vanBeijeren-1975}
(although results have been obtained at zero temperature through
a connection with the dimer model).
The cases of $\beta$ close to $\beta_c$ and that of $n$ not equal to $(0,0,1)$
are open.
The higher-dimensional problem is open,
except that the result of Dobrushin generalises to higher dimensions.

\begin{figure}
	\centering
  \newcommand{\fw}{.4\textwidth}
	\begin{subfigure}{\fw}
		\centering
		\includegraphics[width = 4.5cm]{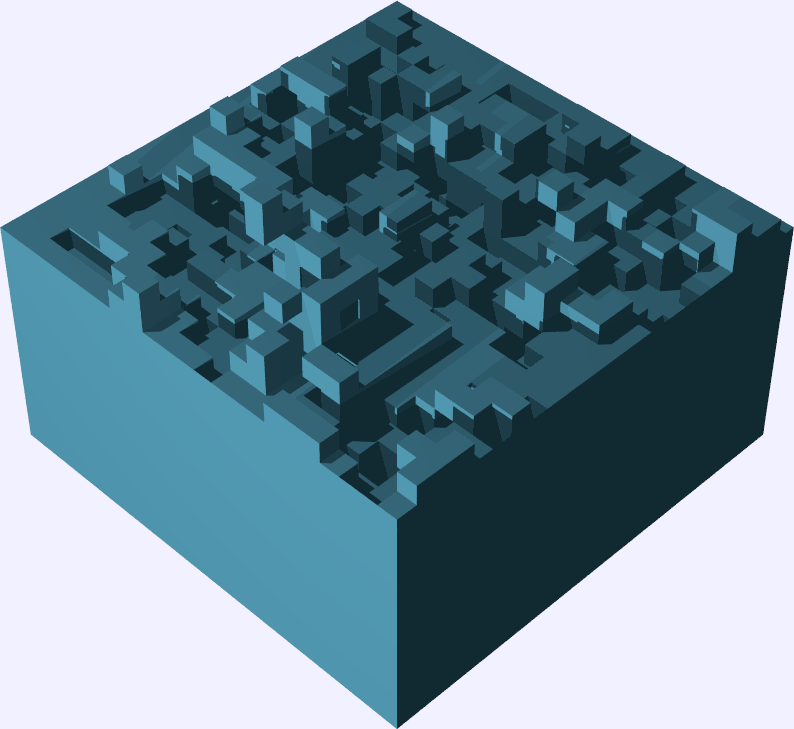}
		\caption*{\(\beta = 0.24\)}
	\end{subfigure}\begin{subfigure}{\fw}
		\centering
		\includegraphics[width = 4.5cm]{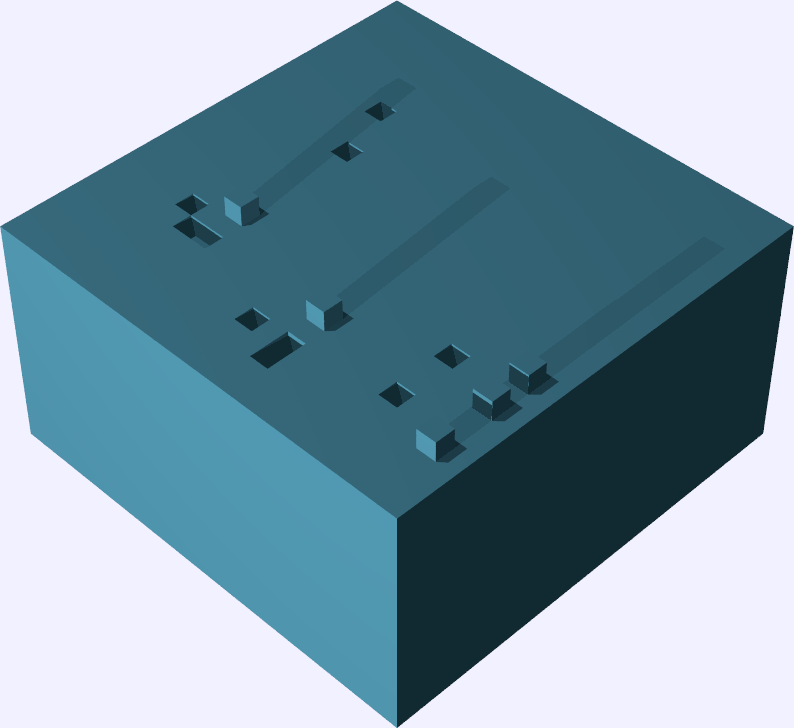}
		\caption*{\(\beta = 0.6\)}
	\end{subfigure}
	\caption{Interface of the Ising model for $d=3$, $n=(0,0,1)$, and $N=22$}
	\label{fig:Ising_interface}
\end{figure}

\subsection{Effective models and integer-valued height functions}

To obtain a better understanding of the problem,
a family of effective models for the interface has been introduced:
that of \emph{integer-valued height functions}.
Integer-valued height functions, or simply \emph{height functions},
are random functions $\phi:\mathbb Z^d\to\mathbb Z$ whose gradient is typically not too large 
and whose graph $\operatorname{Graph}(\phi)\subset\mathbb Z^d\times\mathbb Z$
has a geometrical interpretation as a random surface or interface.
The archetypical example of such a model is the solid-on-solid model,
which will be at the centre of this article.
The solid-on-solid model is obtained from the Ising model
by sending the interaction strength along one of the axes to infinity~\cite{Bricmont+Fontaine+Lebowitz-1982}.
The qualitative behaviour of the interface in the Ising model is expected to
match that of the solid-on-solid model.
In dimension one, effective interfaces are essentially random walks and their macroscopic behaviour 
is thus trivial. 
In dimension three and higher,
localisation was proven rigorously for the solid-on-solid model and the discrete Gaussian model,
using positivity of the transfer operator and reflection positivity through the infrared bound
(see~\cite{Frohlich+Simon+Spencer-1976,Frohlich+Israel+Lieb+Simon-1978} for the discrete Gaussian model
and~\cite{Bricmont+Fontaine+Lebowitz-1982} for the solid-on-solid model).
One should also mention~\cite{Gopfert+Mack-1981} where strong clustering properties (exponential decay of correlations) of the discrete Gaussian in dimension three are derived at all temperatures.


We now turn the discussion towards dimension two,
which corresponds to the three-dimensional Ising model.
To aid the discussion, we shall introduce some informal notation.
The letter \(\mu\) denotes a probability measure 
with the random height function $\phi$ such that the law of its gradient $\nabla\phi$
is shift-invariant in $\mu$ (that is, gradient observables are invariant under lattice translations).
The tilt of the interface is conveniently parametrised by the \emph{slope} of \(\mu\):
the unique functional \(u\in (\mathbb R^2)^*\)
such that
\(\mu(\phi_y-\phi_x) = u(y-x)\) for any $x,y\in\mathbb Z^2$.
For example, the interface normal vector $(0,0,1)$
of the Ising model
corresponds to the zero slope $0\in(\mathbb R^2)^*$.
Delocalisation can be characterised in two distinct ways.
First, one can study the presence of such measures $\mu$ which are Gibbs for the solid-on-solid model.
The model is said to \emph{delocalise} when no shift-invariant Gibbs measures exist.
This form of delocalisation can be described as \emph{soft},
because it is connected to a range of qualitative arguments which are not expected 
to yield quantitative results directly.
We shall prove that the first form of delocalisation coincides with \emph{delocalisation in finite volume},
which asserts that fluctuations of the height at the zero vertex tend to infinity 
as boundary conditions are enforced further and further away from this centre vertex.
The latter form of delocalisation is generally of higher interest and more akin to the notion of delocalisation introduced
for the Ising model.
We shall speak of \emph{quantitative delocalisation} whenever there is an explicit lower bound 
on the fluctuations which is of the expected order.
While we shall not prove quantitative results directly, 
we derive properties which have led to quantitative delocalisation in other models.
Delocalisation is expected to provide the stepping stone towards quantitative delocalisation:
for the model of square ice, for example,
it has been proven that the delocalisation must be logarithmic whenever there is delocalisation~\cite{Duminil-Copin+Harel+Laslier+Raoufi+Ray-2019}.
This result was proven separately from the delocalisation itself,
and in that sense consists of a dichotomy. After a first version of this paper appeared on the arXiv, similar dichotomy results have been obtained by the first author for the class of height functions considered in Section~\ref{sec:abs_FKG}, see~\cite{Lammers-2022}.
The phase diagram of delocalisation in the solid-on-solid model is expected to match that of the Ising model:
the localisation-delocalisation transition occurs at the zero slope $0\in(\mathbb R^2)^*$,
while the model is predicted to delocalise at all temperatures for nonzero slopes.
Compare this to the discrete Gaussian model,
where localisation occurs at the same temperature for each integer slope,
because there exist simple transformations that allow one to modify a measure in order to change its slope~\cite{Shlosman-1983}.
Indeed, start from any Gibbs measure, add an integer-valued, deterministic, harmonic
height function to each sample, and observe that the law so obtained is also Gibbs for the same specification.
In particular, if $\mu$ is a Gibbs measure at slope $0$ and $\psi$ the deterministic function given by
$\phi_x=u(x)$ for $u\in(\mathbb Z^2)^*$, then the resulting Gibbs measure has slope $u$.



Over the past two decades, the study of integer-valued height functions 
has developed into a much broader field.
Several models of two-dimensional height functions 
are linked directly to combinatorics, integrable probability, and complex geometry.
Let us therefore give a (non-exhaustive) overview of some of the results which are related
to the localisation-delocalisation transition.
\begin{enumerate}
	\item
	The most complete picture has been obtained for the dimer model,
	which has an alternative, combinatorial interpretation as a model of random tilings.
	It is known for this model that its scaling limit is the Gaussian free field~\cite{Kenyon-2001}. This result has been extended to interacting dimer models in~\cite{Giuliani+Mastropietro+Toninelli-2017}.
	Moreover, it is known that under a limit in the boundary conditions,
	the local measure converges to the unique shift-invariant gradient Gibbs measure at the correct 
	slope in the case of the hexagonal dimer model~\cite{Aggarwal-2019}.
	Quantitative delocalisation
	(that is, logarithmic delocalisation)
	is indeed implied by these much stronger results.
	\item The general analysis of height functions was initiated by Sheffield.
	In his seminal work \emph{Random surfaces}~\cite{Sheffield-2006} he proved, 
	amongst other things, that height functions on planar graphs delocalise 
	for any slope that is not in the dual lattice of the vertex set of the graph.
	This implies, for example, that the dimer models as well as the square ice delocalise at any non-extremal slope (that is, slopes which do not freeze the configuration).
	For the result on the square ice, we also refer to~\cite{Chandgotia+Peled+Sheffield+Tassy-2020}.
	It was later shown that conditional on this delocalisation,
	the delocalisation is logarithmic:
	the previously mentioned dichotomy result~\cite{Duminil-Copin+Harel+Laslier+Raoufi+Ray-2019}.
	These results have recently been extended to the more general six-vertex model 
	for $a=b=1$ and $1\leq c\leq2$, see~\cite{Duminil-Copin+Karrila+Manolescu+Oulamara-202-}.
	This is notable especially because the six-vertex model consists of non-planar interactions (the heights interact through the diagonals of the square lattice).
	\item Several results are available for quantitative delocalisation of $1$-Lipschitz functions on the triangular
	lattice, which is intimately linked to the Loop~$O(n)$ model with the parameter $n=2$.
	In particular, we mention~\cite{Duminil-Copin-Glazman-Peled-Spinka-2017}
	for Nienhuis' critical point (the parameter $x=1/\sqrt 2$)
	and~\cite{Glazman+Manolescu-2019} for the uniform case.
	\item Quantitative delocalisation for the solid-on-solid and discrete Gaussian model
	for small \(\beta>0\) was derived using a connection to the Coulomb gas,
	utilising a multi-scale expansion~\cite{Frohlich+Spencer-1981}.
	The focus of that article is the Berezinskii-Kosterlitz-Thouless phase transition,
	which is strongly related to the localisation-delocalisation transition.
	The methods of~\cite{Frohlich+Spencer-1981} have recently been extended
	to the case of tilted boundary conditions in~\cite{garban2020statistical}.
	\item Delocalisation in the shift-invariant sense was recently derived for a large class of models on planar graphs which have a maximum degree of three~\cite{Lammers-2021}.
	The only condition is that $\beta$ is sufficiently small;
	for the solid-on-solid and the discrete Gaussian models, one requires $0<\beta\leq\log 2$.
	\item
	After a first draft of the present paper was completed,
	convergence to the Gaussian free field of the discrete Gaussian model in two dimensions
	at very large temperature was obtained in~\cite{bauerschmidt2022discrete}.
	This new result relies on the application of finite range decompositions 
	of Gaussian measures, as well as renormalisation group transformations.
\end{enumerate}

It is worth mentioning that all results are established for planar graphs,
with the exception of the delocalisation result on the six-vertex model,
where the delocalisation relies (in part) on some favourable combinatorial properties of the model.
Extensions to non-planar models without such combinatorial features have not been found to date. 
Let us finally mention that localisation at a high inverse temperature
is generally handled through an adaptation of Dobrushin's argument for the Ising model;
see~\cite{brandenberger1982decay} for the original argument and the survey~\cite{Velenik-2006} for a review and for further references therein.


\subsection{An informal description of the main results}

We prove three main results.
First, we demonstrate that there exist no Gibbs measures with a shift-invariant 
gradient for the solid-on-solid model whenever $\beta>0$ is sufficiently small (Theorem~\ref{theorem:delocalisation_symmetric}
and Theorem~\ref{theorem:delocalisation_asymmetric}).
The proof works on any shift-invariant planar graph,
and does not require one to work on the square lattice.
The result thus implies delocalisation at any slope,
and not exclusively at the zero slope.
Second, we demonstrate that the absolute value of the height 
function satisfies the FKG lattice condition whenever
nonnegative boundary conditions are imposed (Theorem~\ref{thm:FKG_main_results}).
The property holds true for both the solid-on-solid model 
as well as for the discrete Gaussian model,
and may be of independent interest given the recent successes that it has 
had in the context of the square ice and the six-vertex model in 
proving logarithmic delocalisation~\cite{Duminil-Copin+Harel+Laslier+Raoufi+Ray-2019,Duminil-Copin+Karrila+Manolescu+Oulamara-202-}. 
The general strategy of these articles stems from an older work~\cite{Duminil-Copin+Sidoravicius+Tassion-2017}
where the FKG lattice condition is used to prove a dichotomy for the quantitative behaviour 
at and off criticality in the random-cluster model.
We do not prove quantitative results,
and use the new property instead to build a bridge between the two notions of delocalisation.
In other words, we use the property to prove our third main result (Theorem~\ref{thm:fvdelocalisation}),
namely that the fluctuations of the solid-on-solid model tend
to infinity when zero boundary conditions are imposed,
in the limit of the size of the box which is centred around the vertex of interest.
After a first version of this paper appeared online,
the super-Gaussian property (Theorem~\ref{thm:FKG_main_results}) was used in~\cite{van2021elementary}
to provide an alternative derivation of the BKT transition in the XY model 
through its connection to a natural height function fitting the framework of~\cite{Lammers-2021}.

\subsection{Organisation of the article}
Section~\ref{sec:main} contains the formal statements of the main results,
as well as the formal definitions required to state them.
Sections~\ref{section:background}--\ref{section:slope}
are dedicated to proving the first main result (Theorem~\ref{theorem:delocalisation_symmetric}
and Theorem~\ref{theorem:delocalisation_asymmetric}).
Section~\ref{section:background} describes some results from the work of Sheffield,
and introduces some useful constructions.
Sections~\ref{section:height_percolation} and~\ref{section:slope}
contain the proofs of Theorem~\ref{theorem:delocalisation_symmetric}
and Theorem~\ref{theorem:delocalisation_asymmetric} respectively.
The remaining two sections are very different in spirit
and do not rely on the shift-invariant framework.
Section~\ref{sec:abs_FKG} is dedicated to the proof 
of the FKG inequality (Theorem~\ref{thm:FKG_main_results}),
by exploiting a connection with the Ising model.
Section~\ref{sec:finite_vol} explains the connection between the two notions of delocalisation
(Theorem~\ref{thm:fvdelocalisation}), and is fairly general in spirit.

\section{Main results}
\label{sec:main}
\subsection{Height function models}
\label{sec:HF_models}

The purpose of this subsection is to efficiently introduce the mathematical formalism for the analysis of height functions.
The notation is broadly the same as in the works of Georgii~\cite{Georgii-1988} and Sheffield~\cite{Sheffield-2006}.
Most notably, we introduce the \emph{super-Gaussian property}
for potential functions,
which leads to the FKG lattice condition for the absolute value of the height function
in the models induced by such potentials.

\begin{definition}[shift-invariant planar graphs]
	A \emph{shift-invariant planar graph} is a pair $(\mathbb G,\mathcal L)$ where $\mathcal L$ is a lattice in $\mathbb R^2$ and $\mathbb G=(\mathbb V,\mathbb E)$ a connected graph which has a locally finite planar embedding in $\mathbb R^2$ that is invariant under the action of $\mathcal L$.
	The term \emph{maximum degree} has its usual meaning.
\end{definition}

\begin{definition}[height functions]
	A \emph{height function} is an integer-valued function on the vertex set $\mathbb V$.
	Write $\Omega$ for the set of height functions,
	and write $\mathcal F$ for the Borel $\sigma$-algebra of the product topology on $\Omega=\mathbb Z^\mathbb V$.
	Let $\mathcal F^\nabla$ denote the \emph{gradient $\sigma$-algebra},
	that is, the sub-$\sigma$-algebra of $\mathcal F$ generated by the height differences 
	$\phi_y-\phi_x$ with $x$ and $y$ ranging over $\mathbb V$.
	Identify each element $\theta\in\mathcal L$ with the map $\mathbb V\to\mathbb V,\,x\mapsto x+\theta$.
	For $A\in\mathcal F$ and $\theta\in\mathcal L$,
	let $\theta A$ denote the event $\{\phi\circ\theta:\phi\in A\}\in\mathcal F$.
	Write $\mathcal P(\Omega,\mathcal F)$ for the set of probability measures on $(\Omega,\mathcal F)$, and write $\mathcal P_\mathcal L(\Omega,\mathcal F)$ for the set of measures which are furthermore \emph{shift-invariant}, that is,
	measures $\mu$ which satisfy $\mu\circ\theta=\mu$ for any $\theta\in\mathcal L$.
	Write $\mathcal P_{\mathcal L(\nabla)}(\Omega,\mathcal F)$ for the set of measures which have a shift-invariant 
	gradient, that is, the set of measures $\mu$ which satisfy 
	$\mu|_{\mathcal F^\nabla}\circ\theta=\mu|_{\mathcal F^\nabla}$ for any $\theta\in\mathcal L$.
\end{definition}

\begin{definition}[potential functions, Hamiltonians]
\label{definition:Hamiltonian}
Write $\lambda$ for the counting measure on $\mathbb Z$.
A \emph{potential function} is a convex function $V:\mathbb Z\to\mathbb R$ with the property that the \emph{edge transition distribution} given by $e^{-\beta V}\lambda$ is a nontrivial finite measure
for any $\beta>0$.
Such a function is called \emph{symmetric} whenever $V(x)=V(-x)$ for all $x\in\mathbb Z$
and \emph{Lipschitz} whenever $|V(x+1)-V(x)|\leq 1$ for all $x\in\mathbb Z$.
A potential function $V$ is called \emph{super-Gaussian} whenever its second derivative $V^{(2)}:\mathbb Z\to\mathbb R$ defined by
\[
V^{(2)}(k):=V(k+1)-2V(k)+V(k-1)
\] is non-increasing over the nonnegative integers.
A \emph{potential} is a family $V=(V_{xy})_{xy}$ of potential functions indexed by 
the directed edges of the graph $\mathbb G$
which satisfies the following two properties:
\begin{enumerate}
	\item $V_{xy}(z)=V_{xy}(-z)$ for all $xy\in\mathbb E$ and $z\in\mathbb Z$ (\emph{consistency}),
	\item $V_{xy}=V_{(\theta x)(\theta y)}$ for all $xy\in\mathbb E$ and $\theta\in\mathcal L$ (\emph{shift-invariance}).
\end{enumerate}
A potential $V$ is called \emph{symmetric}, \emph{Lipschitz}, or \emph{super-Gaussian} whenever all potential functions are symmetric, Lipschitz, or super-Gaussian respectively.
For any finite set $\Lambda\subset\mathbb V$,
introduce the associated \emph{Hamiltonian} defined by
\[
	H_\Lambda:
	\Omega\to\mathbb R,\,
	\phi\mapsto\sum\nolimits_{xy\in\mathbb E(\Lambda)}V_{xy}(\phi_y-\phi_x),
\]
where $\mathbb E(\Lambda)$ denotes the set of undirected edges in $\mathbb E$ that have at least one endpoint in $\Lambda$.
Note that we must direct each edge for the definition to make sense,
but that the value of each term does not depend on the ordering chosen due to the consistency requirement on the potential $V$.
\end{definition}

The prime example of a symmetric Lipschitz potential is of course the solid-on-solid model,
which is defined by $V_{xy}(\cdot):=|\cdot|$ for any $xy\in\mathbb E$.

\begin{definition}[specifications, Gibbs measures]
\label{def:specifications_gibbs_measures}
Let $\beta$ denote a fixed strictly positive real number,
called the \emph{inverse temperature}.
For any height function $\phi$ and for any finite set $\Lambda\subset\mathbb V$, write 
$\gamma_\Lambda(\cdot,\phi)$ for 
the probability measure 
defined by 
\begin{equation}
\label{eq:spec_def}
	\gamma_\Lambda(\cdot,\phi)
	:=
	\frac1{Z_\Lambda(\phi)}e^{-\beta H_\Lambda}(\delta_{\phi|_{\mathbb V\smallsetminus\Lambda}}\times\lambda^\Lambda),
\end{equation}
where $Z_\Lambda(\phi)$ denotes a suitable normalisation constant.
Thus, to sample from $\gamma_\Lambda(\cdot,\phi)$, set first the random height function equal to $\phi$ on the complement of $\Lambda$,
then sample its values on $\Lambda$ proportionally to $e^{-\beta H_\Lambda}$.
The family $(\gamma_\Lambda)_\Lambda$ with $\Lambda$ ranging over the finite subsets of $\mathbb V$ is a \emph{specification}.
A measure $\mu\in\mathcal P(\Omega,\mathcal F)$ is called a \emph{Gibbs measure}
if $\mu=\mu\gamma_\Lambda$ for all finite $\Lambda\subset\mathbb V$.
Write $\mathcal G$ for the collection of Gibbs measures,
and $\mathcal G_\mathcal L$ for the collection 
of shift-invariant Gibbs measures,
that is, $\mathcal G_\mathcal L:=\mathcal G\cap\mathcal P_\mathcal L(\Omega,\mathcal F)$.
Finally, write $\mathcal G_{\mathcal L(\nabla)}:=\mathcal G\cap\mathcal P_{\mathcal L(\nabla)}(\Omega,\mathcal F)$
for the collection of Gibbs measures with a shift-invariant gradient.
\end{definition}

\subsection{Statement of the main results}

This subsection presents the main results of this article.
We first state two delocalisation results for Lipschitz potentials on planar graphs.
The results are stated in the shift-invariant infinite-volume setting:
this very much in the style of the work of Sheffield~\cite{Sheffield-2006}.
Next, we state a result which asserts that this notion of delocalisation coincides
with the notion of delocalisation in finite volume (with fixed boundary conditions).
The latter notion is generally considered to be of higher interest within the 
physics literature and in the context of the solid-on-solid model.
The gap between the two notions is bridged by another result,
which may be of independent interest:
namely that for super-Gaussian potentials,
the absolute value of the height function satisfies the FKG lattice condition.
This result is stated separately, at the end of this section.

\begin{theorem}[delocalisation]
	\label{theorem:delocalisation_symmetric}
	Let $(\mathbb G,\mathcal L)$ be a shift-invariant planar graph
	of maximum degree $m$, and
	let $V$ denote a symmetric Lipschitz potential.
	Then for any $0<\beta\leq (\log2)/m$,
	the associated height function model delocalises,
	in the sense that the set $\mathcal G_\mathcal L$ is empty.
\end{theorem}

\begin{theorem}[delocalisation at any slope]
	\label{theorem:delocalisation_asymmetric}
	Let $(\mathbb G,\mathcal L)$ be a shift-invariant planar graph
	of maximum degree $m$, and
	let $V$ denote a Lipschitz potential.
	Then for any $0<\beta\leq (\log2)/2m$,
	the associated height function model delocalises at any slope,
	in the sense that the set $\mathcal G_{\mathcal L(\nabla)}$ is empty.
\end{theorem}

In particular, these first two main results apply to the solid-on-solid model at the given temperatures.
We connect these results to delocalisation in finite volume through the following theorem.
For the avoidance of doubt, $\gamma_\Lambda(\cdot,0)$ denotes the 
local measure of the model of interest in the finite subset $\Lambda\subset\mathbb V$
with boundary condition $0\in\Omega$.

\newcommand\TWONOTIONSTHEOREMSTATEMENT{
	Let $(\mathbb G,\mathcal L)$ be a shift-invariant planar graph with some distinguished root vertex $r$,
	and let $V$ denote a symmetric super-Gaussian potential.
	Then $\mathcal G_\mathcal L$ is empty if and only if 
	\begin{equation*}
		\lim_{n\to\infty}\gamma_{\Delta_n}(\phi_r^2,0)=\infty
	\end{equation*}
	for any sequence $(\Delta_n)_{n}$ of finite subsets of $\mathbb V$
	such that $\cup_{n\geq 0}\cap_{k\geq n}\Delta_k=\mathbb V$.
	Moreover, the limit on the left in the display is well-defined and independent of the sequence chosen whenever $V$ satisfies 
	these conditions, whether or not $\mathcal G_\mathcal L$ is empty.
}

\begin{theorem}[delocalisation in finite volume]
	\label{thm:fvdelocalisation}
	\TWONOTIONSTHEOREMSTATEMENT
\end{theorem}

The theorem applies to the solid-on-solid model,
where it asserts that the variance tends to infinity 
for $\beta\leq (\log2)/m$.
The super-Gaussian property is also satisfied by the discrete Gaussian model
(the discrete Gaussian free field conditioned to take integer values).
The previous theorem relies on an important observation:
that the absolute value of the height function $\phi$
satisfies the FKG lattice condition in the measure $\gamma_\Lambda(\cdot,\psi)$
whenever $\psi$ is a height function taking nonnegative values.
This observation does not require that the graph is 
planar or shift-invariant.
Therefore it is also not required that the associated potential family is shift-invariant
(which would simply not make sense for more general graphs).

\newcommand\FKGLATTICESTATEMENT{
	Let $\mathbb G=(\mathbb V,\mathbb E)$ denote a connected locally finite graph,
	and let $\Lambda$ denote a finite, strict subset of $\mathbb V$.
	Consider a symmetric super-Gaussian potential $V$,
	which on this occasion is not required to be shift-invariant.
	Let $\psi\in\Omega$ denote a height function taking nonnegative values only.
	Then the absolute value $|\phi|$ of the height function 
$\phi$ satisfies the FKG lattice condition in the measure $\mu:=\gamma_\Lambda(\cdot,\psi)$,
in the sense that 
\[
	\mu(|\phi|=\xi\vee\zeta)\mu(|\phi|=\xi\wedge\zeta)\geq\mu(|\phi|=\xi)\mu(|\phi|=\zeta)
\] 
for any nonnegative height functions $\xi,\zeta\in\Omega$ which equal $\psi$ on $\mathbb V\smallsetminus\Lambda$.
}

\begin{theorem}[FKG lattice condition for the absolute height]
	\label{thm:FKG_main_results}
\FKGLATTICESTATEMENT
\end{theorem}

Subsection~\ref{subsec:extensions_of_FKG} is dedicated to the extension
of this theorem to real-valued height functions,
as well as models with other \emph{a priori} spin distributions $\lambda$.

\section{Background on height functions}

\label{section:background}

The purpose of this section is to recall some existing results on height functions,
and to introduce notations as well as some simple convenient mathematical objects.
The results cited here are especially useful for the proofs 
of Theorem~\ref{theorem:delocalisation_symmetric} and Theorem~\ref{theorem:delocalisation_asymmetric} in Section~\ref{section:height_percolation} and Section~\ref{section:slope} 
respectively.

\subsection{General notations}

Throughout each section amongst Sections~\ref{section:background}--\ref{section:slope},
the notations $\mathbb G$, $\mathcal L$, $V$, and $\beta$ are fixed,
and---in order to arrive at a contradiction---it is supposed that $\mathcal G_\mathcal L$ is nonempty.
The letter $\mu$ then denotes some fixed shift-invariant Gibbs measure in the set $\mathcal G_\mathcal L$.
It is well-known that \(\mu\) may be decomposed into ergodic components which are also shift-invariant Gibbs measures,
and by doing so and choosing one such component to replace $\mu$,
we may assume without loss of generality that \(\mu\) is itself ergodic.
This means that $\mu(A)\in\{0,1\}$ for any event 
$A$ which satisfies $\theta A=A$
for all $\theta\in\mathcal L$.
Define $\mathcal F_\Lambda:=\sigma(\phi_x:x\in\Lambda)$ for $\Lambda\subset\mathbb V$.

\subsection{Results in the work \emph{Random surfaces} of Sheffield.}
\label{subsection:overview_random_surfaces}

Sheffield demonstrated---amongst many other things---that ergodic measures are also extremal,
in his seminal work on height functions titled \emph{Random surfaces}~\cite{Sheffield-2006}.
This result, as well as two of its corollaries, form the starting point of our analysis.
Remark that these results are not constrained to the two-dimensional setting.
The three formal statements are as follows:
\begin{enumerate}
	\item If $\mu\in\mathcal G_\mathcal L$ is ergodic, then it is also \emph{extremal},
	meaning that $\mu(A)\in\{0,1\}$ for any event $A$ that is measurable with respect 
	to $\mathcal F_\Lambda$ for each cofinite set $\Lambda\subset\mathbb V$,
	\item If $\mu$ is ergodic, then the distribution of $\phi_x$ is log-concave for any $x\in\mathbb V$;
	in particular, the random variable $\phi_x$ has finite mean and moments of all orders,
	\item If $\mu$ is ergodic, then it satisfies the \emph{FKG inequality}, meaning 
	that $\mu(fg)\geq\mu(f)\mu(g)$ for any pair of random variables 
	$f,g:\Omega\to[0,1]$ which are \emph{increasing} in the sense that 
	$f(\psi)\geq f(\phi)$ and $g(\psi)\geq g(\phi)$ for any pair of height functions $\phi,\psi\in\Omega$
	with $\psi\geq \phi$.
\end{enumerate}
It is straightforward to derive the last two statements from the result on extremality.
Write $\Lambda_n$ for the set of vertices at a graph distance at most $n$ from some fixed root vertex $r\in\mathbb V$,
and define the \emph{topology of local convergence} to be the coarsest topology on $\mathcal P(\Omega,\mathcal F)$
that makes the map $\mu\mapsto\mu(A)$ continuous for any finite $\Lambda\subset\mathbb V$ and for any $A\in\mathcal F_{\Lambda}$.
Log-concavity and the FKG inequality hold true for any measure of the form $\gamma_\Lambda(\cdot,\phi)$ because 
each potential function is convex.
The second and third result now follow from the first one simply because $\mu$ equals the limit of 
$\gamma_{\Lambda_n}(\cdot,\phi)$ as $n\to\infty$
for $\mu$-almost every $\phi$ in the topology of local convergence,
due to extremality.

For the result on delocalisation at a slope (Theorem~\ref{theorem:delocalisation_asymmetric}),
we shall appeal to another result of \emph{Random surfaces}.
Consider for this purpose some measure $\nu\in\mathcal G_{\mathcal L(\nabla)}$.
It may be assumed without loss of generality that the gradient of $\nu$ is ergodic,
in the sense that $\nu(A)\in\{0,1\}$ for any gradient event $A\in\mathcal F^\nabla$
which satisfies $\theta A=A$ for all $\theta\in\mathcal L$.
Define the \emph{slope} of $\nu$ to be the unique element $u\in(\mathbb R^2)^*$
such that 
$\nu(\phi_y-\phi_x)=u(y-x)$ for all $x,y\in\mathbb V$ with $y-x\in\mathcal L$;
it is easy to demonstrate that such an element exists.
The \emph{dual lattice} $\mathcal L^*$ of $\mathcal L$ consists of the slopes $u\in(\mathbb R^2)^*$ which take integer values on $\mathcal L$.
Sheffield derived delocalisation for all Gibbs measures with an ergodic gradient whose slope 
does not lie in $\mathcal L^*$.
In other words, we may suppose that the slope of the measure $\nu$ introduced above lies in $\mathcal L^*$.

This is convenient for the following reason.
The measure $\nu$ and the potential $V$ can be transformed in such a way that the slope of 
$\nu$ is the zero slope $0\in(\mathbb R^2)^*$, and therefore $\nu\in\mathcal G_\mathcal L$.
The transformation on the potential preserves the Lipschitz property.
For the proof of Theorem~\ref{theorem:delocalisation_asymmetric}
it therefore suffices to demonstrate that the set $\mathcal G_\mathcal L$ is empty
(under the same hypotheses).
The transformation is straightforward to define.
If $u$ is the slope of $\nu$,
then write $a:\mathbb V\to\mathbb Z$
for some function which satisfies $a(y)-a(x)=u(y)-u(x)$ for any $x,y\in\mathbb V$ with $y-x\in\mathcal L$.
To sample from the transformed measure $\nu'$,
sample first a height function $\phi$ from $\nu$,
then define $\phi':=\phi-a$.
The measure $\nu'$ is a Gibbs measure with respect to the specification $\gamma'$
corresponding to the same inverse temperature $\beta$, now 
with the potential $V'$ defined by 
\[
	V'_{xy}(z):=V_{xy}(z+(a(y)-a(x))).
\] 
Indeed, the potential $V'$ is shift-invariant,
and Lipschitz whenever $V$ is Lipschitz.

\subsection{Phase coexistence on planar graphs}

Consider an ergodic measure $\nu\in\mathcal P_\mathcal L(\Omega,\mathcal F)$ which satisfies the FKG inequality
and such that the random height function $\phi$ takes values in $\{0,1\}$ almost surely.
Informally, we think of $\nu$ as being a site percolation measure.
The measure $\nu$ is said to have \emph{finite energy} if the conditional probability of a vertex having height 
one,
given the heights of all other vertices, is always strictly between zero and one.
It is known that under these conditions it almost never occurs 
that both $\{\phi=0\}$ and $\{\phi=1\}$ have an infinite cluster.
In other words, we may rule out phase coexistence for such site percolation measures.
This result was first proven by Sheffield in~\cite[Chapter~9]{Sheffield-2006},
and we also refer to~\cite[Theorem~1.5]{Duminil-Copin+Raoufi+Tassion-2017} for an alternative proof.

\subsection{Stochastic ordering}

Consider two measures $\mu,\nu\in\mathcal P(\Omega,\mathcal F)$ with height functions $\phi$ and $\psi$ respectively.
Say that $\mu$ is \emph{stochastically dominated} by $\nu$,
and write $\mu\preceq\nu$, if there is a coupling between the two measures 
such that $\phi\leq \psi$.
We shall sometimes say that $\phi$ is stochastically dominated by $\psi$,
and write $\phi\preceq\psi$,
leaving the reference to the corresponding measures implicit.
For fixed $\Lambda\subset\mathbb V$,
say that $\phi$ is stochastically dominated by $\psi$ on $\Lambda$,
and write $\phi|_\Lambda\preceq\psi|_\Lambda$,
whenever $\phi|_\Lambda\leq\psi|_\Lambda$ for some coupling of $\mu$ and $\nu$.
Remark that the specification $\gamma$ is \emph{monotone} in the sense that it preserves the stochastic ordering:
if $\mu\preceq\nu$, then $\mu\gamma_\Lambda\preceq\nu\gamma_\Lambda$
for any finite $\Lambda\subset\mathbb V$.
We shall write $\alpha\in\mathcal P(\Omega,\mathcal F)$
for the measure with height function $X$ such that $(X_x)_{x\in\mathbb V}$ 
is a collection of i.i.d.\ random variables each having a Bernoulli distribution with parameter 
$1/2$ and support $\{0,1\}$
throughout this article.
This measure shall play an important role in our proofs of
Theorem~\ref{theorem:delocalisation_symmetric} and Theorem~\ref{theorem:delocalisation_asymmetric}.

\subsection{Exterior contours}
\label{subsec:ext_contours_def}

Consider a finite set $\Lambda\subset\mathbb V$.
By the \emph{exterior} of $\Lambda$ we mean the unique infinite cluster $E(\Lambda)\subset\mathbb V$ of $\mathbb V\smallsetminus\Lambda$.
The \emph{exterior contour} of $\Lambda$ is the set $\Gamma(\Lambda)$ consisting of the vertices
in $\Lambda$ which are adjacent to $E(\Lambda)$.
The \emph{interior} of $\Lambda$, denoted by $\Delta(\Lambda)$,
consists of the remaining vertices: those not in $E(\Lambda)$ or $\Gamma(\Lambda)$.
Note that $\Delta(\Lambda)$ can be written as the union of all finite clusters 
of the set $\mathbb V\smallsetminus\Gamma(\Lambda)$.
We say that the exterior contour $\Gamma(\Lambda)$ \emph{surrounds} some vertex $x\in\mathbb V$
whenever $x\in\Delta(\Lambda)$.
The exterior contour $\Gamma(\Lambda)$ can informally be thought of as a finite collection of closed $*$-paths 
through the planar graph $\mathbb G$.

Recall that $\Lambda_n$ denotes the set of vertices at a graph distance at most $n$ from the root vertex $r\in\mathbb V$.
Suppose that $n\in\mathbb N$ and that $\Lambda\subset\Lambda_n$ is a random set
such that the event $x\in\Lambda$ is measurable with respect to $\phi_x$ for each $x$.
Then the natural exploration process for $\Gamma(\Lambda)$ reveals only the heights
of the vertices in the set $\Lambda_n\smallsetminus\Delta(\Lambda)$.





\section{Mean height and level set percolation}
\label{section:height_percolation}

In this section, we direct our attention to the proof of Theorem~\ref{theorem:delocalisation_symmetric}.
Recall that $\mathbb G$, $\mathcal L$, $V$ and $\beta$ are fixed,
and that $\mu$ is an ergodic extremal Gibbs measure.
We aim to derive a contradiction, proving that such a measure $\mu$ cannot exist. 
Introduce the events
\begin{equation}
\label{eq:def_perco_heights}
	A_k^{\square} := \{\text{the set $\mathbb V\smallsetminus\{x:\varphi_x\mathrel\square k\}$ does not percolate}\}
\end{equation}
for \(\square\in\{\leq,\geq\}\) and \(k\in\mathbb Z\).
Thus, the event \(A_k^{\square}\) occurs if and only if each vertex of $\mathbb G$ is surrounded by 
the exterior contour of the set \(\{\varphi\mathrel\square k\}\cap\Lambda_n\) for $n$ large enough.
Such events are shift-invariant, and therefore occur with probability zero or one for \(\mu\).
The mean height $\mu(\phi_r)$ at the vertex $r$ is related to the previous family of events through the following lemma.
The lemma relies crucially on the fact that $V$ is a symmetric potential with convex potential functions.

\begin{lemma}
	\label{lemma:exploration}
	Let $\mathbb G$ denote a shift-invariant planar graph with some distinguished root vertex $r$,
	and let $V$ be a symmetric potential.
	Consider any inverse temperature $\beta>0$.
	Suppose that $\mu$ is a Gibbs measure and that $\phi_r$ is integrable.
	If the event \(A_k^\square\) occurs almost surely
	then $\mu(\phi_r)\mathrel\square k$,
	for any $\square\in\{\leq,\geq\}$ and for any \(k\in\mathbb Z\).
\end{lemma}

Its proof is entirely straightforward; see~\cite[Lemma~3.5]{Lammers-2020} for details.
Since the law of the set $\{\phi\mathrel\square k\}$ is shift-invariant and depends in a monotone fashion on $\phi$,
we may almost surely rule out phase coexistence of this set and its complement.
Taking into account the previous lemma,
we conclude that 
there is a unique integer $n$ such that $\mu(\phi_r)=n$
and such that $\mu(A_k^\geq)=1(n\geq k)$ and $\mu(A_k^{\leq})=1(n\leq k)$.
We shall suppose that $n=0$ without loss of generality.
Let us now specialise further to symmetric Lipschitz potentials.
For the following lemma, recall the definition of the measure $\alpha$
with the random height function $X$.

\begin{lemma}
	\label{lemma:domination_simple}
Let $\mathbb G$ denote a shift-invariant planar graph of maximum degree $m$,
let $V$ denote a Lipschitz potential, and fix $0<\beta\leq\log2/m$.
Consider $\Lambda\subset\mathbb V$ finite and let $\nu\in\mathcal P(\Omega,\mathcal L)$
denote any probability measure such that $\nu=\nu\gamma_\Lambda$.
Write $\nu'$ for the measure $\nu$ conditioned on the event that 
$\phi_x\geq 0$ for all $x\in\Lambda$.
Then $\phi|_\Lambda\succeq X|_\Lambda$ in the measures $\nu'$ and $\alpha$
respectively.
\end{lemma}

\begin{proof}
	It suffices to demonstrate that the lemma holds true for $\Lambda$ consisting of a single vertex $x$,
	and for $\nu$ of the form $\gamma_\Lambda(\cdot,\psi)$.
	The full result may be recovered by first averaging over all $\psi$
	and then inducting on the number of vertices in $\Lambda$.
		Write $p_k$ for $\nu(\phi_x=k)$.
		The goal is to prove that $p_0\leq\sum_{k>0}p_k$.
	For any $k$, we have
	\[
		\log\frac{p_{k+1}}{p_k}=-\beta\sum_{y\sim x}(V_{yx}(k+1-\psi_y)-V_{yx}(k-\psi_y))\geq -\beta m\geq -\log 2,
	\]
	where the sum is over the neighbours $y$ of $x$, and where $m$ denotes the maximum degree of the graph $\mathbb G$.
	The first inequality is due to the Lipschitz property of the potential $V$;
	the second inequality is due to the choice of $\beta$.
	In particular, this implies that $p_{k+1}\geq p_k/2$.
	By comparison with a geometric distribution with parameter $1/2$,
	we observe that $p_0\leq\sum_{k>0}p_k$.
	This proves the lemma.
\end{proof}

We now start the proof of delocalisation for symmetric potentials.

\begin{proof}[Proof of Theorem~\ref{theorem:delocalisation_symmetric}]
	Let \(K_n\) denote the event that the exterior contour 
	\[\Gamma_n:=\Gamma(\{\phi\geq 0\}\cap\Lambda_n)\] surrounds \(r\).
	We sometimes write $\Gamma$ for $\Gamma_n$ for brevity.
	Note that the event $K_n$ occurs with high probability since the event $A_0^\geq$ occurs almost surely.
	Write $\mu_n$ for the conditioned measure $\mu(\cdot|K_n)$.
	The expectation $\mu_n(\phi_r)$ can be calculated by first 
	conditioning on the contour $\Gamma$, then averaging over all such contours.
	Formally, write 
	\[
		\mu_n(\phi_r)
		=
		\int\mu_n^{\Gamma}(\phi_r)d\mu_n(\Gamma),
	\]
	where $\mu_n^{\Gamma}$ denotes the conditioned measure.
	If we write $\Delta:=\Delta(\{\phi\geq 0\}\cap\Lambda_n)$ for the corresponding interior, then
	\[
		\mu_n^{\Gamma}=\mu_n^{\Gamma}\gamma_{\Delta},
	\]
	simply because the exploration which leads to the discovery of the contour $\Gamma$
	does not appeal to the height of the vertices in $\Delta$;
	see the end of Subsection~\ref{subsec:ext_contours_def}.
	In other words, we have
	\[
		\mu_n(\phi_r)
		=
		\int(\mu_n^{\Gamma}\gamma_{\Delta})(\phi_r)d\mu_n(\Gamma).
	\]
	The previous lemma asserts that the distribution of the restriction of $\phi$ to $\Gamma$ in the conditioned measure $\mu_n^\Gamma$
	dominates the height function $X|_\Gamma$ in the measure $\alpha$.
	Monotonicity in boundary conditions and the Markov property therefore imply that
	\[
		(\mu_n^{\Gamma}\gamma_{\Delta})(\phi_r)\geq (\alpha\gamma_{\Delta})(X_r) = 1/2,
	\]
	where the final equality is due to symmetry of the potential $V$.
	In particular,
	\[
		\mu(\phi_r|K_n)=\mu_n(\phi_r)
		\geq
		1/2.
	\]
	The event $K_n$ occurs with high probability as $n\to\infty$,
	which proves that $\mu(\phi_r)\geq 1/2$. 
	This contradicts the assertion that $\mu(\phi_r)=0$.
\end{proof}

\section{Delocalisation at any slope}
\label{section:slope}

The following lemma implies Theorem~\ref{theorem:delocalisation_asymmetric}
due to Subsection~\ref{subsection:overview_random_surfaces}.

\begin{lemma}
\label{lemma:alternative_statement}
Under the hypotheses of Theorem~\ref{theorem:delocalisation_asymmetric},
the set \(\mathcal G_\mathcal L\) is empty.
\end{lemma}

All notations remain the same as in the previous section, except that 
$V$ is no longer required to be symmetric, and that we impose that $\beta\leq(\log2)/2m$.
In order to derive 
a contradiction, we suppose that $\mathcal G_\mathcal L$ is not empty,
and thus contains some ergodic extremal measure $\mu$.
We must change our strategy:
Lemma~\ref{lemma:exploration} is false whenever $V$ is not symmetric,
and we cannot say that $\alpha\gamma_\Lambda(X_r)=1/2$.
Rather than considering a single height function $\phi$,
we shall consider here the product measure $\mu^2=\mu\times\mu$,
and write $(\phi,\phi')$ for the random pair of independent height functions in this product measure.
Write $\psi:=\phi'-\phi$ for the height difference.
Note that $\mu^2$ is extremal since it is a product of extremal measures.
Indeed, it is straightforward to verify that any tail event for the product space
satisfies a zero-one law for $\mu^2$ by appealing to Fubini-style integration.
Note furthermore that the pair $(-\phi,\phi')$ satisfies the FKG inequality,
and therefore the function $\psi$---which depends in a monotone fashion on this pair---does so as well.
Obviously $\mu^2(\psi_r)=0$.
As before, introduce the events
\begin{equation}
\label{eq:def_perco_heights_slope}
	A_k^{\square} = \{\text{the set $\mathbb V\smallsetminus\{x:\psi_x\mathrel\square k\}$ does not percolate}\}
\end{equation}
for \(\square\in\{\leq,\geq\}\) and \(k\in\mathbb Z\).
Extremality implies that these events occur with probability zero or one for $\mu^2$.
By symmetry, $\mu(A_0^\leq)=\mu(A_0^\geq)$.
If both events occur almost never, then almost surely the sets $\{\psi>0\}$ and $\{\psi<0\}$
percolate simultaneously, a contradiction. 
Conclude that the event $A_0^\geq$ occurs almost surely.

Introduce a new random function $\zeta:\mathbb V\to\mathbb Z\cup\{\infty\}$ defined by
\[
	\zeta_x:=
	\begin{cases}
		\varphi_x +1 & \text{if $\varphi_x<\varphi_x'$,} \\
		\varphi_x& \text{if $\varphi_x=\varphi_x'$,} \\
		\infty &\text{if $\varphi_x>\varphi_x'$.}
		\end{cases}
\]
Notice that $\{\zeta<\infty\}=\{\psi\geq 0\}$ and that $\phi_x\leq\zeta_x\leq\phi_x'$
for each $x$ such that $\phi_x\leq\phi_x'$.
We shall study the law of the measure $\mu^2$ when conditioned on the values of $\zeta$.
The idea is to compare the law of $\phi$ and $\phi'$ directly 
in this conditioned measure.
We focus on the following lemma before doing so.
Recall the definition of the measure $\alpha$ with the random height function $X$.

\begin{lemma}
	Consider a finite set $\Lambda\subset\mathbb V$,
	and let $\nu$ denote a probability measure on the product space $(\Omega,\mathcal F)\times(\Omega,\mathcal F)$
	which satisfies $\nu=\nu(\gamma_\Lambda\times\gamma_\Lambda)$.
	Write $\nu'$ for the measure $\nu$ conditioned on the event that $\zeta|_\Lambda=a$
	for some $a\in\mathbb Z^\Lambda$.
	(In particular, $\zeta|_\Lambda<\infty$.)
	Then the law of $\psi$ in $\nu'$ dominates the law of $2X$ 
	in the measure $\alpha$
	when both height functions are restricted to $\Lambda$.
\end{lemma}

\begin{proof}
It suffices to prove the lemma for $\Lambda=\{x\}$;
the whole result then follows by induction.
Without loss of generality, $a=0$
and $\nu=\gamma_{\{x\}}(\cdot,\xi)\times\gamma_{\{x\}}(\cdot,\xi')$
for $\xi,\xi'\in\Omega$.
For the lemma, it suffices to prove that
\begin{equation}
	\label{eq:to_do}
	\nu'(\psi\geq 2)\geq 1/2
\end{equation}
where $\nu'$ is the measure $\nu$ conditioned on $\zeta_x=0$.
Note that $\nu'$-almost surely
\[
(\phi_x,\phi_x')
=
\begin{cases}
	(0,0)&\text{if $\psi_x=0$,}\\
	(-1,\psi_x-1)&\text{if $\psi_x>0$.}\\
\end{cases}
\]
Write $p_k:=\nu'(\psi_x=k)$ for all $k\geq 0$.
Since the values of $\phi$ and $\phi'$ are fixed on all 
vertices other than $x$,
we observe that the previous display implies
\[
\log \frac{p_{k+1}}{p_k}\geq -\beta m\geq -(\log2)/2
\]  
by reasoning as in the proof of Lemma~\ref{lemma:domination_simple}.
This implies~\eqref{eq:to_do} through comparison with a geometric distribution of parameter $1-\sqrt 2$.
\end{proof}

If $\psi_x\geq 2$, then $\phi_x<\zeta_x<\phi_x'$.
This implies the following consequence of the previous lemma.
For the final conclusion, we use the fact that $-X+1$ and $X$ have the same law for the measure $\alpha$.

\begin{corollary}
	Assume the setting of the previous lemma.
	Then $\phi|_\Lambda\preceq a-X|_\Lambda$ and $a+X|_\Lambda\preceq\phi'|_\Lambda$
	for the measures $\nu'$ and $\alpha$.
	In particular,
	$\phi|_\Lambda+1\preceq\phi'|_\Lambda$
	in some nontrivial coupling of the measure $\nu'$ with itself.
\end{corollary}

\begin{proof}[Proof of Lemma~\ref{lemma:alternative_statement}]
		Let \(K_n\) denote the event that the exterior contour 
	\[\Gamma_n:=\Gamma(\{\psi\geq 0\}\cap\Lambda_n)\] surrounds \(r\).
	We sometimes write $\Gamma$ for $\Gamma_n$ for brevity.
	Note that the event $K_n$ occurs with high probability for $\mu^2$ since the event $A_0^\geq$ occurs almost surely.
	Write $\mu_n$ for the conditioned measure $\mu^2(\cdot|K_n)$.
	The expectation $\mu_n(\psi_r)$ can be calculated by first 
	conditioning on the contour $\Gamma$ and on $a:=\zeta|_\Gamma$, then averaging over all such contours.
	Formally, write 
	\[
		\mu_n(\psi_r)
		=
		\int\mu_n^{\Gamma,a}(\psi_r)d\mu_n(\Gamma,a),
	\]
	where $\mu_n^{\Gamma,a}$ denotes the conditioned measure.
	If we write $\Delta:=\Delta(\{\psi\geq 0\}\cap\Lambda_n)$ for the corresponding interior, then
	\[
		\mu_n^{\Gamma,a}=\mu_n^{\Gamma,a}(\gamma_{\Delta}\times\gamma_\Delta),
	\]
	simply because the exploration which leads to the discovery of the contour $\Gamma$
	does not appeal to the height of the vertices in $\Delta$.
	In other words, we have
	\[
		\mu_n(\psi_r)
		=
		\int(\mu_n^{\Gamma,a}(\gamma_{\Delta}\times\gamma_\Delta))(\psi_r)d\mu_n(\Gamma,a).
	\]
	The previous corollary says that $\mu_n^{\Gamma,a}$ can be coupled with itself in such a way
	that $\phi|_\Gamma+1\preceq\phi'|_\Gamma$.
	Monotonicity and the Markov property, together with the fact that $\gamma$ is a gradient specification,
	imply that
	\[
		(\mu_n^{\Gamma,a}(\gamma_{\Delta}\times\gamma_\Delta))(\phi_r)+1\leq (\mu_n^{\Gamma,a}(\gamma_{\Delta}\times\gamma_\Delta))(\phi_r').
	\]
	In particular, $(\mu_n^{\Gamma,a}(\gamma_{\Delta}\times\gamma_\Delta))(\psi_r)\geq 1$, and
	\[
		\mu^2(\psi_r|K_n)=\mu_n(\psi_r)\geq 1.
	\]
	The event $K_n$ occurs with high probability as $n\to\infty$,
	which proves that $\mu^2(\psi_r)\geq 1$. 
	This contradicts the observation that $\mu^2(\psi_r)=0$.
\end{proof}

\section{The FKG lattice condition for the absolute height}
\label{sec:abs_FKG}

The purpose of this section is to derive Theorem~\ref{thm:FKG_main_results}.
An important connection with the Ising model is discussed \emph{en route}
in Subsection~\ref{subsec:Ising},
and this connection will also play a role in Section~\ref{sec:finite_vol}.
In order to prove these results, we completely abandon the shift-invariant setting of Sections~\ref{section:background}--\ref{section:slope}.
If the interest of the reader is in the relationship between the two notions 
of delocalisation (Theorem~\ref{thm:fvdelocalisation}),
then one may simply accept Theorem~\ref{thm:FKG_main_results} and skip to Section~\ref{sec:finite_vol}.  
First recall the statement of Theorem~\ref{thm:FKG_main_results}.

\begin{theorem*}
\FKGLATTICESTATEMENT
\end{theorem*}

Assume the setting of this theorem throughout this section.
We shall assume simply that the family $V=(V_{xy})_{xy\in\mathbb E}$
consists of identical potential functions;
the proof is no different for the case that they are not identical.
Abuse notation by simply writing $V$ for each potential function $V_{xy}$.

The section consists of three subsections.
In the first subsection, we work out a formula for the probability under $\mu$
that $|\phi|$ is equal to some specific function.
It turns out that this probability can be decomposed as a product of several factors,
where all factors are of a simple form except for one factor, 
which equals the partition function of a ferromagnetic Ising model with 
$+$ boundary conditions (the latter because the boundary condition \(\psi\) is nonnegative).
This is natural because conditional on $|\phi|$,
the sign of each height $\phi_x$ is governed by a ferromagnetic Ising model
where the interaction strengths between the signs are governed by $|\phi|$.
In the second subsection, we state some simple consequences of the
Griffiths-Kelly-Sherman (GKS) inequality,~\cite{Griffiths-1967,Kelly+Sherman-1968}, which are needed for the 
proof of the theorem.
The final subsection completes the proof.
The different factors of the decomposition are handled separately,
appealing to straightforward calculations (for the simple factors)
and the GKS inequalities (for the partition function that appears).

\subsection{The Ising model}
\label{subsec:Ising}
Let $\phi$ denote a height function which agrees with $\psi$ on the complement of $\Lambda$,
and write $\xi:=|\phi|$.
The idea is to write each height $\phi_x$ as the product of $\xi_x$
with some spin $\sigma_x=\pm1$.
We shall later see that symmetry and convexity of $V$
imply that the law of $\sigma$ conditional on $\xi$
is that of an Ising model, where the coupling constants depend on $\xi$.
The purpose of this subsection is to decompose the relative weight of each 
configuration in such a way that this becomes apparent.
We shall first introduce some convenient notations.
Write $\Sigma(\Lambda)$ for the set of spins $\sigma\in\{-1,1\}^\mathbb V$
which satisfy $\sigma_x=1$ for each $x\in\mathbb V\smallsetminus\Lambda$.
Define $\Sigma(\phi):=\{\sigma\in\Sigma(\Lambda):\phi=\sigma\xi\}$,
and write $z(\phi):=|\{x\in\Lambda:\phi_x=0\}|$.
Note that $|\Sigma(\phi)|=2^{z(\phi)}$.
For $\sigma\in\Sigma(\phi)$, one may write the potential $V(\phi_y-\phi_x)$
on each edge $xy\in\mathbb E$ as
\[
	\frac12 (V(\xi_y-\xi_x)+V(\xi_y+\xi_x))
	+ 
	\frac12
	\sigma_x\sigma_y
	 (V(\xi_y-\xi_x)-V(\xi_y+\xi_x)).
\]
Remark that it is important here that $V$ is an even function.
Recall the meaning of $Z_\Lambda(\psi)$ from Definition~\ref{def:specifications_gibbs_measures},
and write the non-normalised probability $Z_\Lambda(\psi)\mu(\phi)$ of the height function $\phi$
as 
\begin{align*}
	Z_\Lambda(\psi)\mu(\phi)
	&=
	2^{-z(\phi)}\sum_{\sigma\in\Sigma(\phi)} e^{-\sum_{xy\in\mathbb E(\Lambda)}V(\phi_y-\phi_x)}
	\\
	&
	=
	2^{-z(\xi)}
	e^{
		-\frac12
		\sum_{xy\in\mathbb E(\Lambda)}
			(V(\xi_y-\xi_x)+V(\xi_y+\xi_x))
	}
	\\
	&
	\color{white}=\color{black}\quad
	\cdot
	\left[
	\sum\nolimits_{\sigma\in\Sigma(\phi)}
		e^{-\frac12
		\sum_{xy\in\mathbb E(\Lambda)}
			\sigma_x\sigma_y
	 		(V(\xi_y-\xi_x)-V(\xi_y+\xi_x))
	 	}
	\right].
\end{align*}
This decomposition is convenient because the first two factors depend only on $\xi$,
and because the last factor is reminiscent of the Ising model.
Indeed, if one is after an expression for $Z_\Lambda(\psi)\mu(|\phi|=\xi)$
rather than $Z_\Lambda(\psi)\mu(\phi)$,
then one simply replaces $\Sigma(\phi)$ by $\Sigma(\Lambda)$
in the previous display.
In that case, one obtains that $Z_\Lambda(\psi)\mu(|\phi|=\xi)$
equals
\begin{equation}
\label{eq:decomposition}
2^{-z(\xi)}
e^{
	-\frac12
		\sum_{xy\in\mathbb E(\Lambda)}
			(V(\xi_y-\xi_x)+V(\xi_y+\xi_x))
}
Z_\Lambda^+(K^\xi)
,
\end{equation}
where the family $K^\xi=(K^\xi_{xy})_{xy\in\mathbb E}$ is defined by 
\begin{equation}
\label{eq:def_K_xi}
	K^\xi_{xy}:=-\frac12(V(\xi_y-\xi_x)-V(\xi_y+\xi_x))
\end{equation}
and where $Z_\Lambda^+(K^\xi)$ denotes the normalisation constant of the Ising model
in the set $\Lambda$ with boundary conditions $1$ and with the interaction strengths 
provided by $K^\xi$,
that is,
\[
Z_\Lambda^+(K^\xi):=
\sum_{\sigma\in\Sigma(\Lambda)}
		e^{
		\sum_{xy\in\mathbb E(\Lambda)}
			\sigma_x\sigma_yK_{xy}^\xi}.
\]
Note that $K^\xi_{xy}\geq 0$ since $\xi_x,\xi_y\geq 0$ and since $V$ is even and convex.
In other words, the Ising model is ferromagnetic.
Moreover, the family $K^\xi$ is increasing in $\xi$ for the same reasons.
Conclude that the expression in~\eqref{eq:decomposition}
is indeed a product of simple non-interacting factors, together with one factor 
that is the partition function of a ferromagnetic Ising model with $+$ boundary conditions.

\subsection{The GKS inequality}

Let us now state a simple consequence of the GKS inequality.
Let $\nu^\xi$ denote the Ising model corresponding to the partition function $Z_\Lambda^+(K^\xi)$,
that is, $\nu^\xi$ is a probability measure on $\Sigma(\Lambda)$ defined by
\[
	\nu^\xi(\sigma):=\frac1{Z_\Lambda^+(K^\xi)}e^{
		\sum_{xy\in\mathbb E(\Lambda)}
			\sigma_x\sigma_yK_{xy}^\xi}
\]
for each $\sigma\in\Sigma(\Lambda)$.
We shall also write $\langle\cdot\rangle_\xi$
for $\nu^\xi(\cdot)$.
Recall that this measure is a ferromagnetic Ising model with positive boundary conditions;
in particular, the GKS inequality applies.
For any $A\subset\mathbb V$, write $\sigma_A:=\prod_{x\in A}\sigma_x$.
Then the GKS inequalities assert that
\[
	\langle\sigma_A\rangle_\xi\geq 0
	\qquad\text{and}\qquad
	\langle\sigma_A\sigma_B\rangle_\xi\geq \langle\sigma_A\rangle_\xi\langle\sigma_B\rangle_\xi
\]
for any $A,B\subset\mathbb V$. We shall use these inequalities in the following form.

\begin{lemma}[GKS inequalities]
\label{lemma:exp_gks}
For any functions $H,H':\mathbb E(\Lambda)\to\mathbb R_{\geq 0}$, we have
\[
	\langle e^{\sum_{xy}\sigma_{xy}H_{xy}}\rangle_\xi
	\geq 1
	\qquad\text{and}\qquad
	\langle e^{\sum_{xy}\sigma_{xy}(H_{xy}+H_{xy}')}\rangle_\xi
	\geq 
	\langle e^{\sum_{xy}\sigma_{xy}H_{xy}}\rangle_\xi
	\langle e^{\sum_{xy}\sigma_{xy}H_{xy}'}\rangle_\xi,
\]
where all sums are over $\mathbb E(\Lambda)$.
\end{lemma}

\begin{proof}
The exponential function in each integrand can be written as a Taylor series around zero (Maclaurin series).
The desired inequalities then follow from a careful manipulation of the terms and the GKS inequalities. 
\end{proof}

\subsection{The proof of the inequality}
\begin{proof}[Proof of Theorem~\ref{thm:FKG_main_results}]
Consider now directly the inequality in the statement of the theorem.
Observe that $\mu(|\phi|=\xi)>0$ for any function $\xi$ which takes nonnegative values 
and equals $\psi$ on the complement of $\Lambda$, since the Hamiltonian in the definition of $\mu$
takes finite values only.
It is known that under this condition---which is known as \emph{finite energy}---it 
is sufficient to consider functions $\xi$ and $\zeta$ which are equal except perhaps at 
two sites, say $u,v\in\Lambda$; see~\cite[Theorem~2.22]{Grimmett-2006}.
For convenience we introduce the following notation:
the function $\xi$ is fixed from now on,
and for any $a,b\in\mathbb Z_{\geq 0}$, we shall write 
$\xi^{ab}$ for the unique height function which equals $\xi$
except that it takes heights $a$ and $b$ at the vertices $u$ and $v$ respectively.
Thus, it suffices to demonstrate that for any 
$a,a',b,b'\in\mathbb Z_{\geq 0}$
with $a'\geq a$ and $b'\geq b$,
we have 
\[
	\mu(|\phi|=\xi^{a'b'})\mu(|\phi|=\xi^{ab})\geq\mu(|\phi|=\xi^{a'b})\mu(|\phi|=\xi^{ab'}).
\]
In fact, it is straightforward to work out that it suffices to consider the case
that $a'=a+1$ and $b'=b+1$ (using induction to recover the other cases),
and by symmetry we may suppose that $a\leq b$ without loss of generality.
First multiply each probability in the previous equation with 
the positive constant $Z_\Lambda(\psi)$,
and recall~\eqref{eq:decomposition}.
Observe that 
\[
	z(\xi^{a'b'})+z(\xi^{ab})=z(\xi^{a'b})+z(\xi^{ab'}).
\]
Therefore it suffices to demonstrate that 
\[
f(\xi^{a'b'})
Z_\Lambda^+(K^{\xi^{a'b'}})
\cdot
f(\xi^{ab})
Z_\Lambda^+(K^{\xi^{ab}})
\geq
f(\xi^{a'b})
Z_\Lambda^+(K^{\xi^{a'b}})
\cdot
f(\xi^{ab'})
Z_\Lambda^+(K^{\xi^{ab'}})
\]
where $f(\xi):=e^{
	-\frac12
		\sum_{xy\in\mathbb E(\Lambda)}
			(V(\xi_y-\xi_x)+V(\xi_y+\xi_x))
}$.
Note that each factor is strictly positive.
We shall handle the two types of factors separately,
that is, we shall prove that 
\begin{equation}
	\label{eq:desired_inequalities}
	\frac{
		f(\xi^{a'b'})
		f(\xi^{ab})
	}{
		f(\xi^{a'b})
		f(\xi^{ab'})
	}
	\geq 1
	\qquad
	\text{and}
	\qquad
	\frac{
		Z_\Lambda^+(K^{\xi^{a'b'}})
		Z_\Lambda^+(K^{\xi^{ab}})
	}{
		Z_\Lambda^+(K^{\xi^{a'b}})
		Z_\Lambda^+(K^{\xi^{ab'}})
	}
	\geq 1.
\end{equation}
Write $c:=a+b$ and $d:=b-a\geq 0$.
Focus first on the equation on the left.
If one expands $f(\cdot)$ and cancels equal factors,
then it is immediate that the only terms surviving are those corresponding to the 
pair $\{u,v\}$.
This is convenient because the height of each height function is known
at these two vertices.
In particular, one obtains
\begin{align*}
	2\log
	\frac{
		f(\xi^{a'b'})
		f(\xi^{ab})
	}{
		f(\xi^{a'b})
		f(\xi^{ab'})
	}
	&
	=
	-
	(
		V(d)+V(c+2)
		+
		V(d)+V(c)
	)
	\\
	&\color{white}=\color{black}
	+
	(
		V(d-1)+V(c+1)
		+
		V(d+1)+V(c+1)
	)
	\\
	&
	=V^{(2)}(d)-V^{(2)}(c+1)\geq 0.
\end{align*}
The final inequality follows from the fact that 
$0\leq d\leq c+1$ and from the fact that $V^{(2)}$ is non-increasing over $\mathbb Z_{\geq 0}$
by definition of a super-Gaussian potential.
This proves the inequality on the left in~\eqref{eq:desired_inequalities}.

To finish the proof we must demonstrate that the inequality on the right in~\eqref{eq:desired_inequalities}
holds true.
Recall the definition of $K^\xi$ in~\eqref{eq:def_K_xi}
and recall that this family is increasing in $\xi$.
This implies that $K^\zeta-K^\xi$ is a nonnegative function for any $\zeta\geq\xi$,
which plays an essential role in following argument.
The idea is to divide each partition function in the equation on the right in~\eqref{eq:desired_inequalities}
by $Z_\Lambda^+(K^{\xi^{ab}})$.
Each ``normalised'' partition function can then be written 
as the expectation of some appropriate function in the measure $\langle\cdot\rangle_{\xi^{ab}}$.
For any vertex $x\in\mathbb V$, we shall write $\mathcal N_x$ for the set of neighbours of $x$ 
with the vertices $u$ and $v$ removed.
We shall first focus on the ratio $Z_\Lambda^+(K^{\xi^{a'b}})/Z_\Lambda^+(K^{\xi^{ab}})$.
If we define
\[
	H_u(\sigma):=\sum_{x\in\mathcal N_u}\sigma_{xu}(K_{xu}^{\xi^{a'b}}-K_{xu}^{\xi^{ab}}),
\]
then 
\[
	\frac{
		Z_\Lambda^+(K^{\xi^{a'b}})
	}{
		Z_\Lambda^+(K^{\xi^{ab}})
	}
	=
	\left\langle
		e^{H_u+\sigma_{uv}(K_{uv}^{\xi^{a'b}}-K_{uv}^{\xi^{ab}})}
	\right\rangle_{\xi^{ab}}.
\]
Similarly, we have 
\[
	\frac{
		Z_\Lambda^+(K^{\xi^{ab'}})
	}{
		Z_\Lambda^+(K^{\xi^{ab}})
	}
	=
	\left\langle
		e^{H_v+\sigma_{uv}(K_{uv}^{\xi^{ab'}}-K_{uv}^{\xi^{ab}})}
	\right\rangle_{\xi^{ab}},
\]
where the definition of $H_v$ equals that of $H_u$ except that $u$ and $\xi^{a'b}$
are replaced by $v$ and $\xi^{ab'}$ respectively.
For the final remaining partition function we have
\[
	\frac{
		Z_\Lambda^+(K^{\xi^{a'b'}})
	}{
		Z_\Lambda^+(K^{\xi^{ab}})
	}
	=
	\left\langle
		e^{H_u+H_v+\sigma_{uv}(K_{uv}^{\xi^{a'b'}}-K_{uv}^{\xi^{ab}})}
	\right\rangle_{\xi^{ab}}.
\] 
Assume for now the claim that
\begin{equation}
\label{eq:claim}
K_{uv}^{\xi^{a'b'}}-K_{uv}^{\xi^{ab}}
\geq
(K_{uv}^{\xi^{a'b}}-K_{uv}^{\xi^{ab}})
+
(K_{uv}^{\xi^{ab'}}-K_{uv}^{\xi^{ab}}),
\end{equation}
and write $D$ for the value on the left minus the value on the right.
Lemma~\ref{lemma:exp_gks} now asserts that 
\begin{align*}
	\frac{
		Z_\Lambda^+(K^{\xi^{a'b'}})
	}{
		Z_\Lambda^+(K^{\xi^{ab}})
	}
	&
	=
	\left\langle
		e^{H_u+H_v+\sigma_{uv}(K_{uv}^{\xi^{a'b'}}-K_{uv}^{\xi^{ab}})}
	\right\rangle_{\xi^{ab}}
	\\
	&
	\geq
	\left\langle
		e^{H_u+\sigma_{uv}(K_{uv}^{\xi^{a'b}}-K_{uv}^{\xi^{ab}})}
	\right\rangle_{\xi^{ab}}
	\left\langle
		e^{H_v+\sigma_{uv}(K_{uv}^{\xi^{ab'}}-K_{uv}^{\xi^{ab}})}
	\right\rangle_{\xi^{ab}}
		\left\langle
		e^{\sigma_{uv}D}
	\right\rangle_{\xi^{ab}}
	\\
	&
	\geq
	\frac{
		Z_\Lambda^+(K^{\xi^{a'b}})
	}{
		Z_\Lambda^+(K^{\xi^{ab}})
	}
	\frac{
		Z_\Lambda^+(K^{\xi^{ab'}})
	}{
		Z_\Lambda^+(K^{\xi^{ab}})
	},
\end{align*}
where for the final inequality we use the fact that $\langle
		e^{\sigma_{uv}D}
	\rangle_{\xi^{ab}}\geq 1$.
	This implies the desired inequality.
	It suffices to prove the claim that~\eqref{eq:claim} holds true.
	In other words, we must demonstrate that 
	\[
		K_{uv}^{\xi^{a'b'}}
		+
		K_{uv}^{\xi^{ab}}
		\geq
		K_{uv}^{\xi^{a'b}}
		+
		K_{uv}^{\xi^{ab'}}
		.
	\]
	Multiply either side by $-2$ and recall the definition of $K^\xi$ to see that the previous inequality is 
	equivalent to
	\[
		V(d)-V(c+2)
		+
		V(d)-V(c)
		\leq
		V(d-1)-V(c+1)
		+
		V(d+1)-V(c+1)
		.
	\]
	Rewrite this inequality in terms of $V^{(2)}$ to obtain
	\[
		-V^{(2)}(d)\leq V^{(2)}(c+1).
	\]
	This inequality is clearly satisfied since $V$ is convex.
\end{proof}


\subsection{Extensions to other \emph{a priori} spin distributions}
\label{subsec:extensions_of_FKG}

The above is---first and foremost---a discussion of the Hamiltonian.
The reference measure $\lambda$ never really plays a role:
the only property of $\lambda$ that is essential
is the property that $x$ and $-x$ are identically distributed (symmetry).
This observation allows for a number of extensions which we now address.
First, we may extend the theorem to real-valued height functions.
For this, we need to suitably adapt the definition of a super-Gaussian
potential.
The following definition applies to the remainder of this section only.

\begin{definition}[real-valued height functions]
Consider the measurable space $(\mathbb R,\mathcal B)$,
and let $\Omega:=\mathbb R^\mathbb V$ denote the collection 
of real-valued height functions endowed with the product $\sigma$-algebra
$\mathcal F$.
A \emph{symmetric super-Gaussian potential function}
is a convex symmetric function $V:\mathbb R\to\mathbb R$
such that the second derivative of $V|_{(0,\infty)}$ is
non-increasing (and well-defined almost everywhere).
A \emph{symmetric super-Gaussian potential} is a consistent (in the sense of Definition~\ref{definition:Hamiltonian}), shift-invariant
family of symmetric super-Gaussian potential functions.
For each finite $\Lambda\subset\mathbb V$,
the Hamiltonian $H_\Lambda:\Omega\to\mathbb R$ is defined exactly as before.
\end{definition}

We first generalise Theorem~\ref{thm:FKG_main_results} as follows.

\begin{theorem}[generalised FKG lattice condition]
	Let $\mathbb G=(\mathbb V,\mathbb E)$ denote a connected locally finite graph,
	and let $\Lambda$ denote a finite, strict subset of $\mathbb V$.
	Consider a symmetric super-Gaussian potential $V$,
	which on this occasion is not required to be shift-invariant.
	Let $\psi\in\Omega$ denote a height function taking nonnegative values only.
	Write $\Omega(\Lambda)$ for the set of height functions $\phi$
	which equal $\psi$ on the complement of $\Lambda$.
	Consider some nonnegative height function $\xi\in\Omega(\Lambda)$,
	and 
	write $\Omega(\xi)$ for the set of height functions 
	$\phi\in\Omega(\Lambda)$
	 which satisfy $|\phi|=\xi$.
	Define 
	\[
		d(\xi):=\sum_{\phi\in\Omega(\xi)}e^{-H_\Lambda(\phi)}
	\]
	for any nonnegative height function $\xi\in\Omega(\Lambda)$.
	This function then satisfies the FKG lattice condition,
	in the sense that 
\begin{equation}
\label{eq:fkgdensity}
	d(\xi\vee\zeta)d(\xi\wedge\zeta)\geq d(\xi)d(\zeta)
\end{equation}	for any nonnegative height functions $\xi,\zeta\in\Omega(\Lambda)$.
\end{theorem}

\begin{proof}
Suppose that $V$ is some symmetric super-Gaussian potential function.
Then the potential function $V|_\mathbb Z$ is super-Gaussian
in the sense of Definition~\ref{definition:Hamiltonian}.
This proves that~\eqref{eq:fkgdensity} is satisfied
 whenever $\psi$, $\xi$, and $\zeta$ take integer values only.
For exactly the same reasons, one observes that~\eqref{eq:fkgdensity}
holds true whenever these functions take values in $\varepsilon\mathbb Z$
for fixed $\varepsilon>0$.
For any $\phi\in\Omega$ and $\varepsilon>0$, write $\phi^\varepsilon$
for the height function $\varepsilon\lceil\phi/\varepsilon\rceil$.
The general inequality in~\eqref{eq:fkgdensity}
may now be obtained by replacing each height function $\phi$
by $\phi^\varepsilon$ and sending $\varepsilon$ to zero,
noting that each factor $d(\phi^\varepsilon)$ tends to
$d(\phi)$ in this limit.
\end{proof}

This immediately leads to the following corollary.

\begin{corollary}
Assume the setting of the previous theorem,
and let $\lambda$ denote any symmetric $\sigma$-finite measure on $(\mathbb R,\mathcal B)$
such that the normalised probability measure 
\[
	\mu:=
	\frac1{Z}e^{-H_\Lambda}(\delta_{\psi|_{\mathbb V\smallsetminus\Lambda}}\times\lambda^\Lambda)
\]
is well-defined.
Then the function $|\phi|$ satisfies the FKG inequality in $\mu$.
\end{corollary}

\begin{proof}
Note that $d/Z$ is a Radon-Nikodym derivative of the law of $|\phi|$
in the measure $\mu$ with respect to the measure
$\delta_{\psi|_{\mathbb V\smallsetminus\Lambda}}\times(\lambda|_{[0,\infty)})^\Lambda$.
\end{proof}

Of principal interest is of course the Lebesgue measure,
although other choices for $\lambda$ are possible:
amongst them tilted versions of the counting measure on $\mathbb Z$
and of the Lebesgue measure on $\mathbb R$.

\section{Delocalisation in finite volume}
\label{sec:finite_vol}

This final section contains the proof of Theorem~\ref{thm:fvdelocalisation},
which asserts that the two notions of delocalisation are the same 
for symmetric super-Gaussian potentials.
The proof relies strongly on the FKG lattice condition for the absolute height
derived in the previous section.
Recall first the statement of Theorem~\ref{thm:fvdelocalisation}.

\begin{theorem*}
\TWONOTIONSTHEOREMSTATEMENT
\end{theorem*}

\begin{proof}
If $|\phi_x|=0$, then $\phi_x=0$.
Thus, the law of $|\phi|$ has the Markov property, but only for those circuits on which the height is equal to zero.
This is convenient, because $0$ is also minimal amongst the nonnegative numbers.
Write $\mu_\Lambda:=\gamma_\Lambda(\cdot,0)$,
and consider two finite subsets $\Lambda\subset\Lambda'\subset\mathbb V$.
Then the previous two observations, together with the FKG lattice condition for the absolute value,
imply that 
$|\phi|\preceq|\phi'|$ in the measures $\mu_\Lambda$ and $\mu_{\Lambda'}$ respectively.
In particular, the value of $\mu_\Lambda(\phi_r^2)=\mu_\Lambda(|\phi_r|^2)$ is increasing in $\Lambda$,
which proves that the limit in the display in the theorem exists in \([0,\infty]\) and is independent of the sequence chosen.

Suppose that this limit is finite.
Write $S$ for the set of subsequential limits of sequences of the form 
$(\mu_{\Delta_n})_n$ in the topology of local convergence,
where $(\Delta_n)_n$ denotes any sequence of finite subsets of $\mathbb V$
such that $\cup_n\cap_{k\geq n}\Delta_n=\mathbb V$.
Note that $S\subset\mathcal G$ by construction.
The arguments in the previous paragraph actually imply that $S$ is nonempty and that the law of $|\phi|$
is independent of the choice of the measure $\mu\in S$.
Note moreover that the set $S$ is shift-invariant,
in the sense that $\mu\circ\theta\in S$ for any $\mu\in S$ and $\theta\in\mathcal L$.
Write $C(S)$ for the closure of the convex hull of $S$
in the topology of local convergence.
Of course, the law of $|\phi|$ is independent of the choice of $\mu\in C(S)$.
This observation implies that the law of $\phi$ is tight over the set $C(S)$.
As a consequence, we derive that the set $C(S)$ is compact.
Obviously $C(S)\subset\mathcal G$.
Write $\mathcal L(n)$ for the set of vectors in $\mathcal L$ at a Euclidean distance 
at most $n$ from $0$.
Fix $\mu\in S$, write $\mu_n$ for the normalised probability measure
\[
	\frac1Z\sum_{\theta\in\mathcal L(n)}\mu\circ\theta,
\]
and note that $\mu_n\in C(S)$.
Therefore $C(S)$ contains some subsequential limit of the sequence $(\mu_n)_n$,
which is shift-invariant by design. 
This proves that the set $\mathcal G_\mathcal L$ is nonempty.

Suppose on the other hand that the set $\mathcal G_\mathcal L$ is not empty, and let $\mu$ denote the unique ergodic Gibbs 
measure with the property that $\mu(\phi_r)=0$.
In Section~\ref{section:height_percolation} we found that the set $\{\phi<0\}$ does not percolate almost surely
in this measure, and consequently we can find circuits surrounding arbitrarily large finite subsets of $\mathbb V$
 through the set $\{\phi\geq 0\}$ almost surely.
It is therefore easy to see that $|\phi'|\preceq|\phi|$ in the measures $\mu_\Lambda$ and $\mu$ respectively for
any finite $\Lambda\subset\mathbb V$.
In particular, the existence of the measure $\mu$ and log-concavity of the density of $\phi_r$
for this measure imply that the limit in the statement of the theorem
is bounded by $\mu(\phi_r^2)<\infty$.
\end{proof}

Consider now the localised case.
Somewhat mysteriously, we were able to demonstrate
in the previous proof that the set $S$ is nonempty 
and that the law of $|\phi|$ is the same for each measure in $S$,
but it is not at all obvious from the proof that the same holds true for
the distribution of the complete height function $\phi$.
In particular, we could not demonstrate that $S=\{\mu\}$
where $\mu$ is the unique ergodic Gibbs measure with $\mu(\phi_r)=0$,
and that
$\mu_{\Delta_n}\to\mu$ for any suitable sequence $(\Delta_n)_n$
in the topology of local convergence.
It should not come as a surprise that these stronger statements hold true as well.
To this end we provide the following argument, which is more constructive in nature than the proof above.

\begin{lemma}
\label{lemma:constructive_proof}
Assume the setting of Theorem~\ref{thm:fvdelocalisation},
and write $\mu_\Lambda:=\gamma_\Lambda(\cdot,0)$.
Suppose that the set $\mathcal G_\mathcal L$ is not empty.
Then for any sequence $(\Delta_n)_{n}$ of finite subsets of $\mathbb V$
	such that $\cup_{n\geq 0}\cap_{k\geq n}\Delta_k=\mathbb V$,
	we have $\mu_{\Delta_n}\to\mu$ in the topology of local convergence,
	where $\mu$ is the unique ergodic Gibbs measure with $\mu(\phi_r)=0$.
\end{lemma}

\begin{proof}
	Consider the measure $\mu_\Lambda$ conditioned on $|\phi|=\xi$.
	Write $(\sigma_x)_{x\in\mathbb V}$ for the the family of signs of $\phi$,
	flipping a fair, independent coin for each vertex with height zero.
	Recall from Subsection~\ref{subsec:Ising} that the conditional law
	of $\sigma|_\Lambda$ is governed by a ferromagnetic Ising model
	with interaction strengths $K^\xi$.
	Recall also that $K^\xi$ is increasing in $\xi$,
	and observe that $K^\xi_{xy}=0$ whenever $\xi_x=0$ or $\xi_y=0$.

	For this conditional measure, one may exploit the connection with the random-cluster model (the Edwards-Sokal coupling) to sample some percolation configuration
	$\omega$
	such that the law of $\sigma$ is decided by flipping a fair coin for each connected component of $\omega$.
	Note that $\omega$ does not percolate in the measure $\mu_\Lambda$ since almost surely 
	$\omega\subset\mathbb E(\Lambda)$.
	By monotonicity of the random-cluster model, the law of $\omega$ is increasing in $K^{\xi}$, which in turn is increasing in \(\xi\).
	Thus, as the law of $|\phi|$ is increasing in $\Lambda$, we conclude that the joint law of the pair $(|\phi|,\omega)$ is increasing in $\Lambda$ and satisfies the FKG inequality for any $\Lambda$.

	As in the proof of Theorem~\ref{thm:fvdelocalisation}, one observes 
	that the joint law of $(\xi,\omega)$ in $\mu_{\Delta_n}$ tends to a unique limit 
	as $n\to\infty$,
	and this limit is ergodic, extremal, and satisfies the FKG inequality.
	The percolation process $\omega$ has at most one infinite cluster almost surely,
	due to the argument of Burton and Keane.
	Write $\mu$ for the measure defined by flipping a fair, independent coin for each cluster of $\omega$
	in order to to determine the sign of the height function $\phi$.
	Then $\mu$ must equal the unique ergodic Gibbs measure of the statement of this theorem, 
	due to uniqueness.
	In fact, we can prove that $\omega$ does not contain an infinite cluster
	almost surely, since percolation of $\omega$ would imply percolation of
	$\{\phi<0\}$ or $\{\phi>0\}$, which was ruled out almost surely in Section~\ref{section:height_percolation}.
\end{proof}

\section*{Acknowledgements}
The authors thank Hugo Duminil-Copin, Aernout van Enter, and Yvan Velenik for helpful comments and discussions.
The authors thank the anonymous referee for several helpful comments and suggestions for improvement.

The first author was supported by the ERC grant CriBLaM.
The second author was supported by the Swiss NSF through an Early Postdoc.Mobility grant and thanks the Roma Tre University for its hospitality.

\bibliographystyle{amsalpha}
\bibliography{SOS_Deloc_Bib,suggestions_aernout,new_papers}

\end{document}